\documentclass[12pt]{amsart}
\usepackage{fullpage}
\usepackage[T1]{fontenc}
\usepackage{amsmath,amssymb,amsthm,mathtools}
\usepackage{booktabs,array}
\usepackage[colorlinks=true,linkcolor=red,urlcolor=blue,citecolor=blue]{hyperref}
\hypersetup{pdftitle={A Proof of Mongardi's Conjecture on Finite Symplectic Group Actions on Hyperk\"ahler Manifolds of K3[n] Type},
            pdfauthor={Jie Fu, Shihao Wang, Zhiwei Zheng}}

\makeatletter
\def\thm@space@setup{%
  \thm@preskip=2ex \thm@postskip=2ex
}
\makeatother
\newtheorem{theorem}{Theorem~}[section]
\newtheorem{proposition}[theorem]{Proposition~}
\newtheorem{lemma}[theorem]{Lemma~}
\newtheorem{corollary}[theorem]{Corollary~}
\newtheorem{conj}[theorem]{Conjecture~}
\theoremstyle{definition}
\newtheorem{definition}[theorem]{Definition~}

\numberwithin{equation}{section}
\newcommand{\Z}{\mathbb Z}
\newcommand{\Q}{\mathbb Q}
\newcommand{\R}{\mathbb R}
\newcommand{\F}{\mathbb F}
\newcommand{\rk}{\operatorname{rk}}
\newcommand{\cont}{\operatorname{cont}}
\newcommand{\divi}{\operatorname{div}}
\newcommand{\Vor}{\operatorname{Vor}}
\newcommand{\pr}{\operatorname{pr}}

\newcommand{\Co}{\mathrm{Co}}
\newcommand{\Mukai}{\mathrm{II}_{4,20}}
\newcommand{\Leech}{\mathbb{L}}

\title[A proof of Mongardi's conjecture]{A Proof of Mongardi's Conjecture on Finite Symplectic Group Actions on Hyperk\"ahler Manifolds of $K3^{[n]}$ Type}
\author{Jie Fu, Shihao Wang, Zhiwei Zheng}
\date{}
\newcommand{\Addresses}{{%
    \bigskip
    \footnotesize

    J.~Fu, \textsc{Qiuzhen college, Tsinghua University, Beijing, China}\par\nopagebreak
    \textit{Email address}: \texttt{fu-j21@mails.tsinghua.edu.cn}

    \medskip

    S.~Wang, \textsc{Qiuzhen college, Tsinghua University, Beijing, China}\par\nopagebreak
    \textit{Email address}: \texttt{wangshihao25@mails.tsinghua.edu.cn}

    \medskip

    Z.~Zheng, \textsc{Yau Mathematical Sciences Center, Tsinghua University, Beijing, China}\par\nopagebreak
    \textit{Email address}: \texttt{zhengzhiwei@mail.tsinghua.edu.cn}
}}

\begin{document}

\begin{abstract}
Let a finite group act faithfully by symplectic automorphisms on a manifold of $K3^{[n]}$ type for $n\geq2$. We prove Mongardi's conjecture that the strict inequality $\rk(S)+\ell(A_S)<24$ holds for its associated Leech coinvariant lattice $S$. The converse was already proved independently by Huybrechts and Mongardi: any Leech coinvariant lattice $S$ satisfying the strict inequality and of rank at most $20$ is realized by such an action for some $n\geq2$. The key idea is to consider the minimum norms of vectors representing discriminant classes of $S$. The absence of numerical walls gives lower bounds for certain classes, while the structure of Leech lattice implies upper bounds for all discriminant classes. Comparing these bounds reduces the proof to finitely many values of $n$. We exclude the remaining cases using further lattice arguments, together with H\"ohn--Mason's classification of fixed-point sublattices of Leech lattice, thereby proving Mongardi's conjecture.
\end{abstract}

\maketitle
\setcounter{tocdepth}{2}
\tableofcontents

\section{Introduction}

Let $X$ be a hyperk\"ahler manifold of $K3^{[n]}$ type. An automorphism of $X$ is symplectic if it acts trivially on $H^{2,0}(X)$. When $n=1$, $X$ is a K3 surface, and the classification of finite symplectic automorphism groups is due to Mukai \cite{MukaiK3}. 

For $n\geq2$, let $G$ be a finite group acting faithfully on $X$ by symplectic automorphisms, and define its coinvariant lattice by $S_G(X)=\bigl(H^2(X,\Z)^G\bigr)^\perp_{H^2(X,\Z)}$. The lattice $S_G(X)$ is negative definite and contains no vectors of norm $-2$, and $G$ acts trivially on its discriminant group $A_{S_G(X)}$ \cite[Lemma~24]{MongardiClassification}; see also \cite[Theorem~1.1]{FuWangZhengStability} for the latter property. Since $H^2(X,\Z)$ has signature $(3,20)$, the negative-definite lattice $S_G(X)$ has rank at most $20$. Moreover, its primitive embedding into the even unimodular Mukai lattice $\Mukai$ of signature $(4,20)$ gives $\rk(S_G(X))+\ell(A_{S_G(X)})\leq24$, where $\ell(A_L)$ denotes the minimum number of generators of the discriminant group $A_L$.

Let $\Leech$ be the positive-definite Leech lattice, with isometry group $\Co_0=O(\Leech)$. By Lemma~\ref{lem:leech-embedding}, $S_G(X)(-1)$ admits a primitive embedding into $\Leech$. Extending the $G$-action by the identity on the orthogonal complement gives $G\subset\Co_0$ and an associated Leech coinvariant lattice $S=S_G(\Leech)\simeq S_G(X)(-1)$ as $G$-lattices.

Huybrechts \cite[Theorem~0.2]{HuybrechtsConway} and Mongardi \cite[Theorem~29 and Lemma~30]{MongardiClassification} independently proved that every Leech coinvariant lattice $S=S_G(\Leech)$ with $\rk S\leq20$ and $\rk S+\ell(A_S)<24$ is geometrically realized: there exist an integer $n\geq2$ and a manifold $X$ of $K3^{[n]}$ type with a faithful symplectic $G$-action such that $S_G(X)\simeq S(-1)$ as $G$-lattices. Huybrechts's construction uses derived autoequivalences of K3 surfaces preserving a Bridgeland stability condition and their induced actions on moduli spaces of stable objects. Mongardi's approach uses lattice embeddings and the global Torelli theorem.

In terms of the associated coinvariant lattice, the remaining direction of Mongardi's conjecture \cite[Conjecture~32]{MongardiClassification} can be stated as follows.

\begin{conj}[Mongardi]\label{conj:mongardi}
Let a finite group $G$ act faithfully by symplectic automorphisms on a manifold $X$ of $K3^{[n]}$ type, with $n\geq2$. Then its associated Leech coinvariant lattice $S$ satisfies $\rk S+\ell(A_S)<24$.
\end{conj}

H\"ohn and Mason classified finite symplectic automorphism groups on manifolds of $K3^{[2]}$ type and proved the strict inequality for their Leech coinvariant lattices \cite[Theorems~8.6--8.7]{HoehnMasonK3two}. Billi, Muller, and Wawak classified finite stable symplectic birational actions on manifolds of $K3^{[3]}$ type \cite[Theorem~2.4]{BilliMullerWawak}. For regular symplectic actions, their classification together with the absence of numerical walls gives the required inequality for $n=3$; see Section~\ref{sec:low}.

Our main result proves Conjecture~\ref{conj:mongardi} for all $n\geq2$.

\begin{theorem}\label{thm:main}
Let a finite group $G$ act faithfully and symplectically on a manifold $X$ of $K3^{[n]}$ type, $n\geq2$. For its associated Leech coinvariant lattice $S$, one has $\rk S\leq20$ and $\rk S+\ell(A_S)<24$.
\end{theorem}

The key idea is to consider the minimum norms of vectors representing discriminant classes of $S$. Since $\rk S+\ell(A_S)\leq24$, it suffices to exclude equality. Assuming $\rk S+\ell(A_S)=24$, unimodular gluing in the Mukai lattice determines certain nonzero discriminant classes of $S$. The absence of numerical walls gives lower bounds for these classes: a representative of sufficiently small norm would produce a numerical wall in $S_G(X)$. The structure of the Leech lattice implies upper bounds for all discriminant classes, obtained by combining a $G$-averaging argument with projections of short representatives in $\Leech/2\Leech$ onto the real span of $S$.

Comparing these bounds reduces the proof to finitely many values of $n$, namely $n\in\{2,3,4,5,7\}$. The cases $n=2,3$ follow from the classifications above. We exclude the remaining cases using further lattice and parity arguments, together with H\"ohn--Mason's classification of fixed-point sublattices of the Leech lattice \cite{HoehnMason}. For $n=4,7$, this classification restricts the possible fixed lattices, and discriminant-minimum estimates and a rank-four Mukai-complement argument rule out the corresponding geometric realizations. For $n=5$, norm-$8$ frames and coordinate parity yield an upper bound of $4$ for the relevant class of order two, contradicting the lower bound of $6$.

The paper is organized as follows. Section~\ref{sec:geometry} recalls the lattice and wall results used below, the short representatives in $\Leech/2\Leech$, and H\"ohn--Mason's classification of Leech fixed-point lattices. Section~\ref{sec:minima} develops the lower and upper bounds for minimum norms in discriminant classes. Section~\ref{sec:remaining} excludes the remaining cases and proves Theorem~\ref{thm:main} and Corollary~\ref{cor:criterion}.

\medskip\noindent\textbf{Acknowledgements.} As this paper was nearing completion, we learned through email correspondence with Stevell Muller of his independent, as yet unpublished joint project with Simon Brandhorst, which contains a proof of Mongardi's conjecture. We thank Muller for this information; the two projects were carried out independently of each other.

\medskip\noindent\textbf{AI disclosure.} Large language models were used to explore proof strategies, assist in the development and refinement of mathematical arguments, and draft and revise the text. The authors have checked all mathematical arguments and the full text, and take responsibility for all results and arguments presented in this paper.

\section{Preliminaries}\label{sec:geometry}

We recall some terminology and results needed below. Nothing in this section is new.

\subsection{Lattices and discriminant forms}

A lattice is a free abelian group of finite rank equipped with a nondegenerate symmetric bilinear form $(\cdot,\cdot)$ with values in $\Q$. Say a lattice is integral if the bilinear form takes its values in $\Z$. Unless stated otherwise, lattices in this paper are integral and even. The norm of a vector $x$ is $x^2=(x,x)$. For $a\in\Q^\times$, write $L(a)$ for $L$ with its bilinear form multiplied by $a$.

For an integral lattice $L$, its dual lattice is $L^\vee=\{x\in L\otimes\Q:(x,L)\subset\Z\}$. The finite abelian group $A_L=L^\vee/L$ is its discriminant group. Write $\det L$ for the determinant of a Gram matrix of $L$, so $|A_L|=|\det L|$. The length $\ell(A_L)$ is the minimal number of generators of $A_L$.

For a nonzero integral lattice $L$, its content is $\cont(L)=\gcd\{(x,y):x,y\in L\}$. The Smith factors of $L$ relate the discriminant length to its content:

\begin{lemma}\label{lem:content}
Let $L$ be a nonzero integral lattice. Then $\rk L=\ell(A_L)$ if and only if $\cont(L)>1$. Furthermore, $\rk L$ and $A_L$ determine $\cont(L)$.
\end{lemma}

\begin{proof}
Let $r=\rk L$, and let $s_1\mid\cdots\mid s_r$ be the positive Smith factors of $L$, including unit factors. Then $\cont(L)=s_1$, and $\ell(A_L)$ equals the number of nonunit factors. Moreover, $A_L$ determines the nonunit factors. Thus $\ell(A_L)=\rk L$ if and only if $s_1>1$, and $\rk L$ and $A_L$ determine $s_1$.
\end{proof}

For an even lattice $L$, its discriminant form is $q_L:A_L\to\Q/2\Z$, given by $q_L(x+L)=x^2+2\Z$ for $x\in L^\vee$.

A sublattice $L\subset M$ is primitive if $M/L$ is torsion-free. A nonzero vector $v\in M$ is primitive if $\Z v$ is a primitive sublattice of $M$.

For a lattice $L$ with a $G$-action, write $L^G$ for its fixed lattice and $S_G(L)=(L^G)^\perp_L$ for its coinvariant lattice.

For a primitive embedding $L\hookrightarrow M$ into an even unimodular lattice, let $R=L^\perp_M$. The embedding determines an isomorphism $\gamma:A_L\xrightarrow{\sim}A_R$ satisfying $q_R\circ\gamma=-q_L$ \cite[Proposition~1.6.1]{Nikulin}. If $G$ acts on $L$ and $R$, these actions extend to $M$ if and only if $\gamma$ intertwines their induced actions on the discriminant groups \cite[Corollary~1.5.2]{Nikulin}. In particular, a $G$-action on $L$ extends to $M$ by the identity on $R$ if and only if it acts trivially on $A_L$.

\subsection{The Leech lattice}\label{sec:leech}

A root of an even positive-definite (negative-definite) lattice is a vector whose norm equals $2$ ($-2$). The Leech lattice $\Leech$ is, up to isometry, the unique even positive-definite unimodular rootless lattice of rank $24$ \cite{ConwaySloane}.

Following \cite[Definition~2.5.1]{MongardiThesis}, we use the positive-definite convention: a Leech pair $(S,G)$ consists of an even positive-definite rootless lattice $S$ and a finite group $G\subset O(S)$ acting faithfully, with $S^G=0$ and trivial action on $A_S$.

The Leech embedding lemma was originally due to Gaberdiel, Hohenegger, and Volpato \cite[Appendix~B.2]{GaberdielHoheneggerVolpato}. The primitiveness was established for $\rk S\leq21$ by Marquand and Muller \cite[Corollary~4.19]{MarquandMuller}, and in full generality by Zheng \cite[Lemma~1.1]{ZhengLeech}.

\begin{lemma}[Leech embedding lemma]\label{lem:leech-embedding}
Let $(S,G)$ be a Leech pair with $\rk S+\ell(A_S)\leq24$. Then $S$ admits a primitive embedding into $\Leech$, under which the $G$-action extends to $\Leech$ by the identity on the orthogonal complement of $S$.
\end{lemma}

For a finite subgroup $G\subset O(\Leech)$, we call $S_G(\Leech)$ a Leech coinvariant lattice.

Let $S=S_G(\Leech)$. Since $S$ is primitive in the unimodular lattice $\Leech$, one has $\rk(\Leech^G)=24-\rk S$ and $\ell(A_{\Leech^G})=\ell(A_S)$. Set $\alpha(S)=24-\rk S-\ell(A_S)=\rk(\Leech^G)-\ell(A_{\Leech^G})\geq0$. Thus $\alpha(S)>0$ is equivalent to $\rk(\Leech^G)>\ell(A_{\Leech^G})$.

H\"ohn and Mason classified the $290$ $\Co_0$-orbits of fixed-point lattices in $\Leech$ and tabulated their Gram matrices and known inclusions \cite[Theorem~1.1, Table~1, and Supplementary Tables~2--3]{HoehnMason}.

The following lemma due to Conway \cite{ConwayLeech} describes the shortest representatives of nonzero classes in $\Leech/2\Leech$. It will be used to bound minimum norms in discriminant classes.

\begin{lemma}\label{lem:modtwo}
Every nonzero class in $\Leech/2\Leech$ has exactly one of the following types of shortest representatives:
\begin{enumerate}
\item one pair of opposite vectors of norm $4$;
\item one pair of opposite vectors of norm $6$;
\item twenty-four pairs of opposite vectors of norm $8$, with vectors from distinct pairs orthogonal.
\end{enumerate}
\end{lemma}

\begin{proof}
See \cite{ConwayLeech} and \cite[Chapter~12, Theorem~2]{ConwaySloane}.
\end{proof}

\subsection{Coinvariant lattices and wall divisors}\label{sec:coinvariant-wall}

For $n\geq2$, let
\[
 L_n=U^{\oplus3}\oplus E_8(-1)^{\oplus2}\oplus\langle2-2n\rangle,
 \qquad \Mukai=U^{\oplus4}\oplus E_8(-1)^{\oplus2}.
\]
Here $U$ is the even unimodular hyperbolic plane and $E_8$ is the positive-definite root lattice of type $E_8$.

Let $X$ be a manifold of $K3^{[n]}$ type. The lattice $H^2(X,\Z)$, equipped with the Beauville--Bogomolov--Fujiki form, is isometric to $L_n$. Choose a primitive embedding $H^2(X,\Z)\hookrightarrow\Mukai$ as in Markman's construction \cite[Corollary~9.5]{MarkmanSurvey}. Let $v$ generate its orthogonal complement. Then $v^2=2n-2$ and the image of $H^2(X,\Z)$ is $v^\perp$.

\begin{definition}\label{def:numerical-wall}
Let $v\in\Mukai$ be primitive with $v^2=2n-2$, where $n\geq2$. For a primitive vector $D\in v^\perp$ with $D^2<0$, let $\Pi=(\Q v+\Q D)\cap\Mukai$. The following equivalent conditions define the numerical wall condition for $D$ relative to $v$:
\begin{enumerate}
\item\label{item:wall-short} There is $\rho\in\Pi\setminus\Q v$ such that $\rho^2\geq-2$ and $0\leq(\rho,v)\leq\frac{v^2}{2}$.
\item\label{item:wall-decomposition} Either there is $s\in\Pi$ with $s^2=-2$ and $0\leq(s,v)\leq\frac{v^2}{2}$, or $v=\rho+\eta$ for $\rho,\eta\in\Pi$ with $\rho^2,\eta^2\geq0$ and $(\rho,v),(\eta,v)>0$.
\item\label{item:wall-numerical} There is $\rho\in\Pi$ with either $\rho^2=-2$ and $0\leq(\rho,v)\leq\frac{v^2}{2}$, or $0\leq\rho^2v^2<(\rho,v)^2\leq\bigl(\frac{v^2}{2}\bigr)^2$.
\end{enumerate}
We also call a rational line in $v^\perp$ a numerical wall if its primitive integral generator satisfies these conditions.
\end{definition}

Condition~\eqref{item:wall-short} is the wall-class criterion of Bayer--Macr\`i \cite[Theorem~12.1]{BayerMacri}. Condition~\eqref{item:wall-decomposition} gives an equivalent formulation of the criterion in \cite[Theorem~5.7]{BayerMacri}. Condition~\eqref{item:wall-numerical} corrects a typographical error in \cite[Definition~9]{MongardiClassification}: the middle inequality should be strict and the final one non-strict.

The following proposition is \cite[Theorem~10 and Lemma~24(3)]{MongardiClassification}. It is a consequence of the wall theory developed by Bayer--Macr\`i \cite{BayerMacri}, Bayer--Hassett--Tschinkel \cite{BayerHassettTschinkel} and Mongardi \cite{MongardiCones}. We sketch its proof.

\begin{proposition}\label{prop:geometry}
Let $X$ be a manifold of $K3^{[n]}$ type, $n\geq2$. If an algebraic class $D\in H^2(X,\Z)$ satisfies Definition~\ref{def:numerical-wall} for the Mukai embedding of $X$, then $D^\perp\cap\mathcal K_X=\varnothing$, where $\mathcal K_X$ is the K\"ahler cone. In particular, if $G$ is a finite group of symplectic automorphisms of $X$, then $S_G(X)$ contains no numerical wall.
\end{proposition}

\begin{proof}
The projective case is a direct consequence of the Mori cone theorem of Bayer--Hassett--Tschinkel \cite[Theorem~1]{BayerHassettTschinkel}. Indeed, suppose that $X$ is projective and choose $\rho$ as in condition~\eqref{item:wall-short}. Since $\rho\in\Q v+\Q D$ and $D\perp v$, its projection $r=\rho-\frac{(\rho,v)}{v^2}v$ is a nonzero rational multiple of $D$. Under Beauville--Bogomolov duality, $r$ corresponds to a nonzero class $C\in H_2(X,\Q)$. Choose an ample class $h$ with $(h,D)\neq0$, and choose the sign of $C$ so that it pairs positively with $h$. The theorem then places $C$ in the Mori cone. Every K\"ahler class has positive intersection with $C$, so it cannot be orthogonal to $D$.

For general $X$, suppose that $\omega\in D^\perp\cap\mathcal K_X$, and let $\sigma$ be a holomorphic symplectic form. The hyperk\"ahler metric with K\"ahler class $\omega$ gives a twistor family whose period planes range over all oriented two-planes in the positive three-space $W=\langle\operatorname{Re}\sigma,\operatorname{Im}\sigma,\omega\rangle_\R$ \cite[Sections~1.13 and~1.17]{HuybrechtsBasicResults}. Since $D$ is algebraic and $(D,\omega)=0$, we have $D\perp W$. Thus the transported class $D_t$ is of type $(1,1)$ on every fibre, and the K\"ahler class $\omega_t\in W$ induced by the metric satisfies $(D_t,\omega_t)=0$.

This family contains a projective fibre $Y$. Indeed, choose $h\in H^2(X,\Z)$ with $h^2>0$. Since $W^\perp$ is negative definite, $h\notin W^\perp$, so $W\cap h^\perp$ is a positive two-plane. On a fibre with this period plane, the class $h$ has type $(1,1)$, so the fibre is projective by \cite[Theorem~2]{HuybrechtsProjectivity}. Parallel transport extends to an isometry of the Mukai lattice \cite[Theorem~9.8]{MarkmanSurvey}, so it preserves the numerical wall condition. Hence $D_Y$ is a numerical wall, while $\omega_Y\in D_Y^\perp\cap\mathcal K_Y$, contradicting the projective case.

For the final assertion, $S_G(X)$ is algebraic because $G$ fixes $H^{2,0}(X)$. Choose a K\"ahler class $\omega_0$ and let $\omega=\frac1{|G|}\sum_{g\in G}g^*\omega_0$. The K\"ahler cone is convex, so $\omega\in\mathcal K_X$. Moreover, $\omega$ is $G$-invariant and therefore orthogonal to $S_G(X)$. If a primitive class $D\in S_G(X)$ were a numerical wall, the preceding argument would give $D^\perp\cap\mathcal K_X=\varnothing$. This contradicts $\omega\in D^\perp\cap\mathcal K_X$.
\end{proof}

Now let $G$ be a finite group of symplectic automorphisms of $X$, and put $M=S_G(X)$. As an orthogonal complement, $M$ is primitive in $H^2(X,\Z)$. It is algebraic and negative definite of rank at most $20$ \cite[Lemma~24(1)--(2)]{MongardiClassification}. Proposition~\ref{prop:geometry} implies that $M$ is rootless, since any vector in $M$ of norm $-2$ is a numerical wall relative to $v$.

The action of $G$ on the discriminant group $A_M$ is trivial \cite[Lemma~24(4)]{MongardiClassification}; an alternative proof was provided in \cite[Theorem~1.1]{FuWangZhengStability}. Since $M^G=0$ and $G$ acts faithfully on $M$, the pair $(M(-1),G)$ is a Leech pair. Moreover, the primitive embedding of $M$ into the unimodular Mukai lattice of rank $24$ gives $\rk M+\ell(A_M)\leq24$.

Lemma~\ref{lem:leech-embedding} therefore gives a primitive embedding $M(-1)\hookrightarrow\Leech$ and extends the $G$-action by the identity on its orthogonal complement. Its image is the Leech coinvariant lattice $S=S_G(\Leech)$, so $M\simeq S(-1)$ as $G$-lattices.

\section{Bounds on Minimum Norms in Discriminant Classes}\label{sec:minima}

\begin{definition}\label{def:discriminant-minimum}
For an integral positive-definite lattice $L$ and a class $\xi\in A_L$, let $\mu_L(\xi)=\min_{x\in\xi}x^2$ denote the minimum norm in $\xi$.
\end{definition}

In this section, we derive lower and upper bounds for $\mu_S$ to restrict the possible values of $n$ when $\alpha(S)=0$.

\subsection{Lower bounds from numerical walls}\label{sec:lower}

We first construct the discriminant class used in the lower bound.

Throughout this subsection, let $S$ be an even positive-definite lattice, let $v\in\Mukai$ be primitive with $v^2=2n-2$, where $n\geq2$, and let $M\simeq S(-1)$ be a primitive sublattice of $v^\perp$. Set $R=M^\perp_{\Mukai}$, and let $a=\divi_R(v)$ be the divisibility of $v$ in $R$, the positive generator of $(v,R)\subset\Z$.

Unimodular gluing gives an isometry $\gamma:(A_R,q_R)\xrightarrow{\sim}(A_M,-q_M)$ \cite[Proposition~1.6.1]{Nikulin}. Identify $A_M$ with $A_S$, and let $c\in A_S$ be the image of $\frac{v}{a}+R$ under $\gamma$.

\begin{lemma}\label{lem:gluing-class}
Suppose additionally that $S=S_G(\Leech)$ for a finite subgroup $G\subset O(\Leech)$. Then the following hold:
\begin{enumerate}
\item\label{eq:divisibility} $\cont(R)=\cont(\Leech^G)$, $\cont(\Leech^G)\mid a$ and $a\mid(2n-2)$.
\item\label{eq:glueclass} The class $c$ has order $a$ and satisfies $q_S(c)\equiv\frac{2n-2}{a^2}\pmod{2\Z}$.
\end{enumerate}
If $\alpha(S)=0$, then $\cont(\Leech^G)>1$ and $a>1$.
\end{lemma}
\begin{proof}
The embedding $M\subset v^\perp$ implies $\rk S\leq20$, so $\Leech^G$ is nonzero. Unimodular gluing gives $A_R\simeq A_S\simeq A_{\Leech^G}$ and $\rk R=\rk(\Leech^G)$, so Lemma~\ref{lem:content} yields $\cont(R)=\cont(\Leech^G)$. Since $v\in R$, we have $\cont(\Leech^G)\mid a$; also $a\mid v^2=2n-2$. Primitivity of $v$ gives the order in~\eqref{eq:glueclass}, and the isometry of discriminant forms gives its quadratic value. If $\alpha(S)=0$, then $\ell(A_{\Leech^G})=\rk(\Leech^G)$, so $\cont(\Leech^G)>1$ by Lemma~\ref{lem:content}.
\end{proof}

The following proposition adapts Mongardi's argument for producing numerical walls \cite[Proposition~31]{MongardiClassification}, giving lower bounds for minimum norms in individual discriminant classes.

\begin{proposition}\label{prop:depth}
If $M$ contains no numerical wall relative to $v$ and $a>1$, then $\mu_S(kc)\geq4+\frac{k^2(2n-2)}{a^2}$ for $1\leq k\leq\lfloor a/2\rfloor$.
\end{proposition}

\begin{proof}
Fix $k$ satisfying $1\leq k\leq\lfloor a/2\rfloor$. Under the identification $A_M=A_S$, choose $y\in M^\vee$ representing $kc$ with $-y^2=\mu_S(kc)$. Since $v$ is primitive in $R$, the class $c$ has order $a$, so $kc\neq0$ and $y\neq0$.

Recall that $\gamma:(A_R,q_R)\xrightarrow{\sim}(A_M,-q_M)$ is the isometry determined by the unimodular gluing of $M$ and $R$ in $\Mukai$. The equality $\gamma(\frac{k}{a}v+R)=y+M$ gives $\rho=\frac{k}{a}v+y\in\Mukai$. Then $\rho^2=\frac{k^2v^2}{a^2}-\mu_S(kc)\in2\Z$ and $(\rho,v)=\frac{k}{a}v^2$.

Let $D$ be a primitive generator of $M\cap\Q y$. Since $M$ is primitive in $\Mukai$, so is $D$. Moreover, $D^2<0$ and $\rho\in\Q v+\Q D$ with $\rho\notin\Q v$.

Suppose $\mu_S(kc)<4+\frac{k^2v^2}{a^2}$. Then $\rho^2>-4$, so evenness of $\Mukai$ gives $\rho^2\geq-2$. Since $0<(\rho,v)=\frac{k}{a}v^2\leq\frac{v^2}{2}$, Definition~\ref{def:numerical-wall} makes $D$ a numerical wall in $M$, a contradiction. Therefore $\mu_S(kc)\geq4+\frac{k^2(2n-2)}{a^2}$.
\end{proof}

\subsection{Upper bounds from the Leech lattice}\label{sec:range}

We bound $\mu_S$ by projecting shortest representatives of classes in $\Leech/2\Leech$ onto $S_\R$. The required estimate follows from a comparison of the Voronoi cells of $S$ and $S^\vee$.

Let $Q:\Leech\to\Z/2\Z$ be given by $Q(z)=\frac{z^2}{2}\bmod2$. It is constant on each coset of $2\Leech$.

For a positive-definite lattice $L$, not necessarily integral, write $\Vor(L)$ for its closed Voronoi cell centered at zero (see \cite[Chapter~2, \S~1.2]{ConwaySloane}) defined by
\[
 \Vor(L)=\{x\in L_\R:2(x,z)\leq z^2\text{ for all }z\in L\}.
\]
The identity $(x-z)^2-x^2=z^2-2(x,z)$ shows that $x\in\Vor(L)$ precisely when $x$ has minimum norm in $x+L$.

If $G$ fixes no vector of $L$ and acts trivially on $A_L$, averaging gives the following comparison:

\begin{lemma}\label{lem:averaging}
Let $L$ be an integral positive-definite lattice, and let $G\subset O(L)$ be a finite subgroup acting trivially on $A_L$ with $L^G=0$. Then $\Vor(L)\subseteq2\Vor(L^\vee)$.
\end{lemma}

\begin{proof}
Assume $x\in\Vor(L)$ and $y\in L^\vee$. Since $G$ acts trivially on $A_L$, one has $y-gy\in L$ and hence $2(x,y-gy)\leq(y-gy)^2$. Since $\sum_g gy\in L_\R^G$, we have $\sum_g gy=0$. As $(gy)^2=y^2$, averaging the above inequality over $G$ gives
\[
 2(x,y)
 =\frac1{|G|}\sum_{g\in G}2(x,y-gy)
 \leq\frac1{|G|}\sum_{g\in G}(y-gy)^2
 =2y^2.
\]
Thus $2(\frac{x}{2},y)\leq y^2$ for every $y\in L^\vee$, which implies $\frac{x}{2}\in\Vor(L^\vee)$. The conclusion follows.
\end{proof}

For a finite subgroup $G\subset O(\Leech)$, put $S=S_G(\Leech)$ and write $\pr_S$ and $\pr_{\Leech^G}$ for the orthogonal projections of $\Leech_\R$ onto $S_\R$ and $(\Leech^G)_\R$, respectively. Applying Lemma~\ref{lem:averaging} to $(S,G)$ gives the following bound.

\begin{corollary}\label{cor:projection}
Let $G\subset O(\Leech)$ be finite, let $S=S_G(\Leech)$, and let $x\in S^\vee$ satisfy $x^2=\mu_S([x])$. There exists $u\in\Leech$ with $\pr_S(u)=x$. For every such $u$ and every $w\in u+2\Leech$, one has $\mu_S([x])=x^2\leq\pr_S(w)^2\leq w^2$.
\end{corollary}

\begin{proof}
We have $S^G=0$ and $S^\perp_\Leech=\Leech^G$. Since $G$ fixes $\Leech^G$ pointwise, its action on $A_S$ is trivial \cite[Corollary~1.5.2]{Nikulin}. Unimodular gluing makes $\pr_S:\Leech\to S^\vee$ surjective \cite[Proposition~1.6.1]{Nikulin}, so a lift $u$ exists. Since $x$ is the shortest in $x+S$, it lies in $\Vor(S)$. Lemma~\ref{lem:averaging} applied to $(S,G)$ gives $x/2\in\Vor(S^\vee)$. For $w\in u+2\Leech$, write $\pr_S(w)=x+2y$ with $y\in S^\vee$. The Voronoi inequality gives $(x,y)\geq-y^2$, hence $\pr_S(w)^2\geq x^2$. Orthogonal projection gives $\pr_S(w)^2\leq w^2$.
\end{proof}

Combining Corollary~\ref{cor:projection} with Lemma~\ref{lem:modtwo}, we have the upper bounds for $\mu_S$; moreover, if $\mu_S(d)>6$
for some $d\in A_S$, then $\cont(\Leech^G)$ must be even.

\begin{lemma}\label{lem:leech-bounds}
Let $G\subset O(\Leech)$ be finite and let $S=S_G(\Leech)$. Then for any $d\in A_S$ one has $\mu_S(d)\leq\max\{6,\frac{\rk S}{3}\}$. Moreover, if $\Leech^G\neq0$ and $\cont(\Leech^G)$ is odd, then $\mu_S(d)\leq 6$.
\end{lemma}

\begin{proof}
Fix $d\in A_S$ and choose $x\in S^\vee$ representing $d$ with $x^2=\mu_S(d)$. Corollary~\ref{cor:projection} gives a lift $u\in\Leech$ with $\pr_S(u)=x$ and $x^2\leq\pr_S(w)^2\leq w^2$ for every $w\in u+2\Leech$.
If $u\in2\Leech$, taking $w=0$ gives $x=0$. If $u\notin2\Leech$, Lemma~\ref{lem:modtwo} says that the class $u+2\Leech$ contains either a vector $w$ with $w^2\in\{4,6\}$ or twenty-four pairs $\{\pm w_i\}_{i=1}^{24}$ of shortest vectors satisfying $(w_i,w_j)=8\delta_{ij}$. In the first case, Corollary~\ref{cor:projection} gives $x^2\leq\pr_S(w)^2\leq w^2\leq6$. In the second case, the vectors $w_i/\sqrt8$ form an orthonormal basis of $\Leech_\R$. Applying~Corollary~\ref{cor:projection} to each $w_i$ gives
\[
 24x^2\leq\sum_{i=1}^{24}\pr_S(w_i)^2
 =\sum_{i=1}^{24}(\pr_S(w_i),w_i)
 =8\operatorname{tr}(\pr_S)=8\rk S.
\]

Assume $x^2>6$. We have to show that $\cont(\Leech^G)$ is even. For every $t\in \Leech^G$, one has $\pr_S(u+t)=x$, so Corollary~\ref{cor:projection} and Lemma~\ref{lem:modtwo} give $6<x^2\leq\min_{w\in u+t+2\Leech}w^2=8$. So, we have $Q(u+t)=0$ for every $t\in \Leech^G$. Hence, for $t_1,t_2\in \Leech^G$,
\[
  (t_1,t_2)\equiv
  Q(u+t_1+t_2)+Q(u+t_1)+Q(u+t_2)+Q(u)=0\pmod2.
\]
Thus $\cont(\Leech^G)$ is even, which implies the conclusion.
\end{proof}

For classes of order three, the norm-$6$ case of Lemma~\ref{lem:modtwo} sharpens the bound in Lemma~\ref{lem:leech-bounds}:

\begin{lemma}\label{lem:order-three-bound}
Let $G\subset O(\Leech)$ be finite and let $S=S_G(\Leech)$. If $d\in A_S$ has order three, then $\mu_S(d)\neq6$. If additionally $\cont(\Leech^G)$ is odd, then $\mu_S(d)<6$.
\end{lemma}

\begin{proof}
Suppose $d\in A_S$ has order three and $\mu_S(d)=6$. Choose $x\in S^\vee$ representing $d$ with $x^2=6$. Then $3x\in S$. Since $Q(3x)=\frac{(3x)^2}{2}\equiv1\pmod2$, Lemma~\ref{lem:modtwo} gives a vector $w\in3x+2\Leech$ of norm $6$. Since $w\in3x+2\Leech$, we have $\pr_S(w)\in3x+2S^\vee=x+2S^\vee$. Lemma~\ref{lem:averaging} and the minimality of $x$ give $\pr_S(w)^2\geq x^2=6$. Since $w^2=6$, the orthogonal decomposition $w=\pr_S(w)+\pr_{\Leech^G}(w)$ forces $\pr_{\Leech^G}(w)=0$, hence $w\in S$. Since $S$ is primitive in $\Leech$, the congruence $w-3x\in2\Leech$ now gives $w-3x\in2S$. Choose the sign of $w$ so that $(w,x)\geq0$. Then $\frac{w-x}{2}\in x+S$ and $\bigl(\frac{w-x}{2}\bigr)^2=3-\frac{(w,x)}{2}\leq3$, a contradiction. If $\cont(\Leech^G)$ is odd, Lemma~\ref{lem:leech-bounds} also gives $\mu_S(d)\leq6$. The equality case is excluded above, so $\mu_S(d)<6$.
\end{proof}

When $\cont(\Leech^G)$ is odd, the function $Q$ takes the value one on a vector of $\Leech^G$. Lemma~\ref{lem:modtwo} then restricts $\cont(\Leech^G)$.

\begin{lemma}\label{lem:odd-content}
Let $G\subset O(\Leech)$ be a finite subgroup. If $\Leech^G\neq0$ and $\cont(\Leech^G)$ is odd, then $\Leech^G$ contains a vector of norm $6$ and $\cont(\Leech^G)\in\{1,3\}$.
\end{lemma}

\begin{proof}
Let $S=S_G(\Leech)$. Because $\cont(\Leech^G)$ is odd, there are $t_1,t_2\in \Leech^G$ with $(t_1,t_2)$ odd, and hence $Q(t_1+t_2)=Q(t_1)+Q(t_2)+1$. Since $\Leech^G$ is even, each of $t_1,t_2,t_1+t_2$ has odd pairing with a vector of $\Leech^G$. The above identity shows that one of these vectors, say $t$, satisfies $Q(t)=1$.

By Lemma~\ref{lem:modtwo}, the class of $t$ in $\Leech/2\Leech$ contains a unique pair $\pm w$ of norm $6$. The group $G$ preserves the pair $\{\pm w\}$. If the $G$-action on this pair is nontrivial, the average of $w$ would vanish, giving $w\in S$. But $t-w\in2\Leech$ would then imply $(t,\Leech^G)\subset2\Z$, contrary to the odd pairing for $t$. Hence $w\in \Leech^G$. Its norm is divisible by the odd content of $\Leech^G$, which must be $1$ or $3$.
\end{proof}

For a Leech coinvariant lattice $S$ and $x\in S^\vee$ representing a class of order two in $A_S$, one has $2x\in S$. We study the class $2x+2\Leech$ using Lemma~\ref{lem:modtwo}.

\begin{lemma}\label{lem:half-parity}
Let $G\subset O(\Leech)$ be finite and let $S=S_G(\Leech)$. If $d\in A_S$ has order two and $q_S(d)\equiv\frac32\pmod{2\Z}$, then $\mu_S(d)=\frac32$.
\end{lemma}

\begin{proof}
Choose $x\in S^\vee$ with $[x]=d$, so $2x\in S$. We have $x^2\equiv\frac32\pmod{2\Z}$, which gives $Q(2x)=2x^2\equiv1\pmod2$. Hence the class $2x+2\Leech$ contains a unique pair $\pm w$ of norm $6$. The group preserves this pair because $2gx-2x=2(gx-x)\in2S$. If some element of $G$ sends $w$ to $-w$, averaging gives $w\in S$. Primitivity then gives $2x-w\in2S$, so $\frac{w}{2}$ represents $d$. Its norm is $\frac32$. Every representative $y$ of $d$ satisfies $y^2\equiv q_S(d)=\frac32\pmod{2\Z}$, so positive definiteness gives $y^2\geq\frac32$. Hence $\mu_S(d)=\frac32$.

Otherwise $G$ fixes $w$, so $w\in\Leech^G$ and $x\perp w$. The integral Leech vector $x-\frac{w}{2}$ would have norm $x^2+\frac32\equiv1\pmod{2\Z}$, contrary to evenness.
\end{proof}

The next bound again starts with the unique norm-$6$ pair in a mod-two class. In the remaining case, Corollary~\ref{cor:projection} and Lemma~\ref{lem:modtwo} give a norm-$8$ frame.

\begin{lemma}\label{lem:two-exception}
Let $G\subset O(\Leech)$ be finite and let $S=S_G(\Leech)$. If $d\in A_S$ has order two and $q_S(d)=\frac12$, then $\mu_S(d)\leq\frac92$.
\end{lemma}

\begin{proof}
Suppose $\mu_S(d)>\frac92$. Lemma~\ref{lem:leech-bounds} gives $\mu_S(d)\leq8$, since $\rk S\leq24$, and the congruence $\mu_S(d)\equiv\frac12\pmod{2\Z}$ forces $\mu_S(d)=\frac{13}{2}$. Choose $x\in S^\vee$ with $[x]=d$ and $x^2=\mu_S(d)=\frac{13}{2}$. The vector $2x\in S$ has norm $26$, so $Q(2x)=13\equiv1\pmod2$ and its mod-two class has a unique short pair $\pm w$ of norm $6$. As $2gx-2x=2(gx-x)\in2S$, the group preserves the pair $\{\pm w\}$. If some element of $G$ sends $w$ to $-w$, averaging gives $w\in S$ and primitivity gives $2x-w\in2S$. Then $\frac{w}{2}$ represents $d$ with norm $\frac32$, a contradiction. Thus $G$ fixes $w$, so $w\in\Leech^G$.

Since $2x-w\in2\Leech$, the vector $\lambda_1=x-\frac{w}{2}$ belongs to $\Leech$ and has norm $8$. Since $w\in\Leech^G$, one has $\pr_S(\lambda_1)=x$. Corollary~\ref{cor:projection} shows that every vector in $\lambda_1+2\Leech$ has norm at least $\frac{13}{2}$. Thus this class has no representative of norm less than $8$. Lemma~\ref{lem:modtwo} shows that its shortest vectors form a frame. Extend $\lambda_1$ to orthogonal frame vectors $\lambda_1,\ldots,\lambda_{24}$ of norm $8$. Let $t_i=\pr_{\Leech^G}(\lambda_i)$. Then $\pr_S(\lambda_i)^2\geq\frac{13}{2}$ and $t_i^2=8-\pr_S(\lambda_i)^2\leq\frac32$. The integers $b_i=(\lambda_i,w)$ are odd, since $(\lambda_1,w)=-3$ and $\lambda_i\equiv\lambda_1\bmod2\Leech$. By Cauchy's inequality, $b_i^2=(t_i,w)^2\leq6t_i^2\leq9$. The frame identity gives $\sum_i b_i^2=8w^2=48$.

Since each $b_i$ is odd and $|b_i|\leq3$, the equality $\sum_i b_i^2=48$ implies that exactly three indices $i$ satisfy $|b_i|=3$. Changing the signs of these frame vectors if necessary, we may assume $b_i=-3$ for each of them. Equality in Cauchy's inequality forces $t_i=-\frac{w}{2}$. Choose $j\ne1$ with $b_j=-3$. The congruence $\lambda_1\equiv\lambda_j\pmod{2\Leech}$ and the equality $t_1=t_j$ imply that $z=(\lambda_1-\lambda_j)/2$ belongs to $S$. Frame orthogonality gives $z^2=4$ and $(x,z)=(\lambda_1,z)=4$. Thus $x-z\in x+S$ has norm $\frac52$, contradicting minimality.
\end{proof}

\subsection{Possible values of \texorpdfstring{$n$}{n}}

Combining Proposition~\ref{prop:depth} with the Leech estimates above, we obtain the following restriction on $n$ when $\alpha(S)=0$:

\begin{proposition}\label{prop:dimension}
Let $G\subset O(\Leech)$ be a finite subgroup, let $S=S_G(\Leech)$, and suppose $\rk S\leq20$ and $\alpha(S)=0$. Let $v\in\Mukai$ be primitive with $v^2=2n-2$, where $n\geq2$, and let $M\simeq S(-1)$ be a primitive sublattice of $v^\perp$ containing no numerical wall relative to $v$. Let $R=M^\perp_{\Mukai}$ and $a=\divi_R(v)$. Then $n\in\{2,3,4,5,7\}$. For $n=4,7$ one has $a=\cont(\Leech^G)=3$.
\end{proposition}

\begin{proof}
Lemma~\ref{lem:gluing-class} gives $\cont(\Leech^G)>1$, $\cont(\Leech^G)\mid a$, and $a\mid(2n-2)$; let $c\in A_S$ be the gluing class of order $a$.

If $a$ is even, the class $d=\frac{a}{2}c$ has order two. Proposition~\ref{prop:depth} and Lemma~\ref{lem:leech-bounds} give $4+\frac{n-1}{2}\leq\mu_S(d)\leq\frac{20}{3}$. Lemma~\ref{lem:gluing-class}\eqref{eq:glueclass} also gives $\mu_S(d)\equiv\frac{n-1}{2}\pmod{2\Z}$. Thus $n\leq6$. If $n=4$, then $q_S(d)=\frac32$, so Lemma~\ref{lem:half-parity} gives $\mu_S(d)=\frac32$, contradicting the lower bound $\frac{11}{2}$. If $n=6$, then $q_S(d)=\frac12$, so Lemma~\ref{lem:two-exception} gives $\mu_S(d)\leq\frac92$, contradicting the lower bound $\frac{13}{2}$. Hence $n=2,3,5$.

If $a$ is odd, then $\cont(\Leech^G)$ is odd, so Lemma~\ref{lem:odd-content} gives $\cont(\Leech^G)=3$. Hence $3\mid a$, and $a\mid(n-1)$. Applying Proposition~\ref{prop:depth} with $k=\frac{a-1}{2}$ and Lemma~\ref{lem:leech-bounds} gives $4+\frac{(n-1)(a-1)^2}{2a^2}\leq6$. Since $n-1\geq a$, this implies $(a-1)^2\leq4a$; together with $3\mid a$ and oddness, it forces $a=3$. The class $c$ has order three and $\cont(\Leech^G)=3$ is odd, so Lemma~\ref{lem:order-three-bound} gives $4+\frac{2n-2}{9}\leq\mu_S(c)<6$. Thus $n<10$ and $3\mid(n-1)$, so $n=4,7$ and $a=\cont(\Leech^G)=3$.
\end{proof}

\section{The Remaining Cases}\label{sec:remaining}

Let $G$ be a finite group acting by symplectic automorphisms on a manifold $X$ of $K3^{[n]}$ type, with $n\geq2$. As in \S\ref{sec:coinvariant-wall}, put $M=S_G(X)$ and let $S=S_G(\Leech)\simeq M(-1)$ be its associated Leech coinvariant lattice. By Proposition~\ref{prop:geometry}, $M$ contains no numerical wall, so Propositions~\ref{prop:depth} and~\ref{prop:dimension} apply.

\subsection{The cases \texorpdfstring{$n=2,3$}{n=2,3}}\label{sec:low}

For $n=2,3$, we use the classifications of symplectic automorphism groups in
\cite{HoehnMasonK3two,BilliMullerWawak}.

\begin{proposition}\label{prop:low}
Let $G$ be a finite group acting by symplectic automorphisms on a manifold $X$ of $K3^{[n]}$ type, and let $S$ be its associated Leech coinvariant lattice. If $n=2$ or $3$, then $\alpha(S)>0$.
\end{proposition}

\begin{proof}
For $n=2$, this is \cite[Theorem~8.6]{HoehnMasonK3two}.

For $n=3$, suppose for contradiction that $\alpha(S)=0$, and let $H=H^2(X,\Z)$ and $M=S_G(X)$. By \cite[Theorem~1.1]{FuWangZhengStability}, $G$ acts trivially on $A_H$. Thus the action is stable in the sense of \cite[\S~1.1]{BilliMullerWawak}. Passing to the stable saturation preserves the fixed lattice and hence the embedded coinvariant lattice $M\subset H$ \cite[\S~1.1]{BilliMullerWawak}. The lattice classification in \cite[\S~2, Theorem~2.4]{BilliMullerWawak} therefore identifies this embedded lattice, under a marking $H\simeq L_3$, with an entry of \cite[Appendix~A, Table~A1]{BilliMullerWawak}.

The numerical part of each Id label corresponds to a row of \cite[Table~1]{HoehnMason}, while the letters distinguish different primitive embeddings into $L_3$. For every listed entry whose underlying H\"ohn--Mason Id has $\alpha=0$, the first number in the $\#(-12,-36)$ column is positive. As explained in \cite[Appendix~A]{BilliMullerWawak}, this number counts classes in the embedded coinvariant lattice of norm $-12$ and divisibility $2$ in the ambient lattice $L_3$. Since $M\simeq S(-1)$ and $\alpha(S)=0$, it follows that $M$ contains a vector $D$ with $D^2=-12$ and $\divi_H(D)=2$.

The vector $D$ is primitive: otherwise, its divisibility would force $D=2E$ with $E^2=-3$, contradicting the evenness of $H$. In the Mukai embedding $H=v^\perp\subset\Mukai$, we have $v^2=4$ and $A_H\simeq A_{\Z v}\simeq\Z/4$. Thus $D/2+H$ and $v/2+\Z v$ are the unique elements of order two in their respective discriminant groups. Unimodular gluing gives $\rho=\frac{v+D}{2}\in\Mukai$ with $\rho^2=-2$ and $(\rho,v)=2$. Hence $D$ is a numerical wall by Definition~\ref{def:numerical-wall}, contradicting Proposition~\ref{prop:geometry}.
\end{proof}

\subsection{The cases \texorpdfstring{$n=4,7$}{n=4,7}}\label{sec:ternary}

Assume $n=4$ or $7$ and $\alpha(S)=0$. Proposition~\ref{prop:dimension} gives $\cont(\Leech^G)=3$, so the $3$-primary part of its discriminant group has length equal to its rank, while the $2$-primary part has smaller length. Together with $\rk(\Leech^G)\geq4$, these conditions leave only rows $22$, $114$, $138$, and $144$ of \cite[Table~1]{HoehnMason}; see the proof of Lemma~\ref{lem:ternary-reduction} below.

We begin with the Leech coinvariant lattice associated with row $22$. In the Niemeier $E_8^{\oplus3}$ model for this row \cite[\S~4]{HoehnMason}, $S_3$ permutes the three factors. Its coinvariant lattice consists of the sum-zero triples, so it is isometric to $A_2\otimes E_8$. We first bound its discriminant minima; this bound will exclude row $22$ and, through the inclusion relations in the H\"ohn--Mason classification, rows $138$ and $144$.

\begin{lemma}\label{lem:a2e8-minima}
Every class $d\in A_{A_2\otimes E_8}$ has $\mu_{A_2\otimes E_8}(d)\leq\frac{10}{3}$.
\end{lemma}

\begin{proof}
Let $K=A_2\otimes E_8$ and realize it as $\{(z_1,z_2,z_3)\in E_8^3:z_1+z_2+z_3=0\}$. The map $e\mapsto[\frac13(2e,-e,-e)]$ identifies $E_8/3E_8$ with $A_K$. We use the standard facts that $\Vor(E_8)$ is defined by $|(x,r)|\leq1$ for all roots $r\in E_8$, and that every $x\in\Vor(E_8)$ has $x^2\leq1$ \cite[Theorem~5 and its corollary, and \S~III-E]{ConwaySloaneVoronoi}.

Choose a shortest representative $e$ of the class in $E_8/3E_8$ corresponding to $d$. Then $e/3\in\Vor(E_8)$, so $e^2\leq9$; evenness gives $e^2\leq8$. If $e^2\leq4$, the representative $\frac13(2e,-e,-e)$ has norm at most $\frac83$.

Otherwise $e^2=6$ or $8$, so $(e/2)^2>1$ and $e/2\notin\Vor(E_8)$. Hence there is a root $r$ with $(e,r)>2$, and integrality gives $(e,r)\geq3$. Minimality of $e$ gives $0\leq(e-3r)^2-e^2=18-6(e,r)$, so $(e,r)=3$. Put $f=e-3r$. Then $f\equiv e\pmod{3E_8}$, $f^2=e^2$, and $(e,f)=e^2-9$. The vector $\frac13(-e,-f,e+f)\in K^\vee$ represents $d$ and has norm $\frac{2e^2}{3}-2\leq\frac{10}{3}$.
\end{proof}

Lemma~\ref{lem:a2e8-minima} and Proposition~\ref{prop:depth} exclude the row-$22$ coinvariant lattice. The inclusion data then reduce the remaining possibilities to one embedding.

\begin{lemma}\label{lem:ternary-reduction}
Let $G\subset O(\Leech)$ be finite, and set $S=S_G(\Leech)$ and $T=\Leech^G$. Suppose $\cont(T)=3$ and, for some primitive $v\in\Mukai$ with $v^2>0$, there is a primitive embedding $S(-1)\hookrightarrow v^\perp$ whose image contains no numerical wall relative to $v$. Then $T$ belongs to the row-$114$ orbit in \cite[Table~1]{HoehnMason}; in particular, $T\simeq D_4(3)$.
\end{lemma}

\begin{proof}
Let $M$ be the image of this embedding, and write $v^2=2n-2$, with $n\geq2$. Since $M$ is negative definite and $v^\perp$ has signature $(3,20)$, we have $\rk S\leq20$ and hence $\rk T\geq4$. Put $R=M^\perp_{\Mukai}$ and $a=\divi_R(v)$. Let $c\in A_S$ correspond to $\frac{v}{a}+R\in A_R$ under unimodular gluing, as in \S\ref{sec:lower}.
Replace $G$ by the full pointwise stabilizer of $T$ to consult the table; this changes neither $T$ nor $S$. Among the entries of \cite[Table~1]{HoehnMason} with $\rk T\geq4$, those whose $3$-primary discriminant group has length $\rk T$ and whose $2$-primary part has smaller length are precisely
\[
\begin{array}{c|c|c}
 \text{row}&\rk T&A_T\\\hline
 22&8&(\Z/3)^8\\
 114&4&(\Z/3)^2\oplus(\Z/6)^2\\
 138&4&(\Z/3)^2\oplus(\Z/9)^2\\
 144&4&(\Z/3)^2\oplus(\Z/6)^2 .
\end{array}
\]
The two length conditions follow from $\cont(T)=3$: a Gram matrix is zero modulo three but not modulo two.

For row $22$, we have $S\simeq A_2\otimes E_8$ and $A_S\simeq(\Z/3)^8$. Lemma~\ref{lem:gluing-class} gives $3\mid a$, while the order of $c\in A_S$ gives $a\mid3$, so $a=3$. Proposition~\ref{prop:depth} then gives $\mu_S(c)\geq4+\frac{2n-2}{9}>\frac{10}{3}$, contradicting Lemma~\ref{lem:a2e8-minima}. Hence row $22$ is excluded. The inclusions in \cite[Supplementary Table~3, rows~138 and~144]{HoehnMason} give $S_{22}\subset S_{138}$ and $S_{22}\subset S_{144}$, where $S_j$ denotes the Leech coinvariant lattice of row $j$.
These containments are primitive because each $S_j$ is primitive in $\Leech$. A primitive embedding of either $S_{138}(-1)$ or $S_{144}(-1)$
into $v^\perp$ whose image contains no numerical wall restricts to a primitive embedding of $S_{22}(-1)$ whose image also contains no numerical wall. Since the fixed lattice of row $22$ has content $3$ and $A_{S_{22}}\simeq(\Z/3)^8$, Lemma~\ref{lem:gluing-class} shows that $v$ has divisibility $3$ in the orthogonal complement of this restricted lattice in $\Mukai$. The preceding row-$22$ argument then excludes rows $138$ and $144$.

Only row $114$ remains.
\end{proof}

The fixed lattice $T$ in row $114$ is isometric to $D_4(3)$ \cite[Supplementary Table~2, row~114]{HoehnMason}. Its normalizer index is $1$ \cite[Table~1, row~114]{HoehnMason}. Thus every isometry of $T$ extends to an isometry of $\Leech$ preserving $T$ and its orthogonal complement.

For the remaining fixed lattice $T=\Leech^G\simeq D_4(3)$, use the standard model $D_4=\{u\in\Z^4:\sum_j u_j\in2\Z\}$, with coordinate vectors $e_1,\ldots,e_4$. The map $u\mapsto u/3$ identifies $D_4/3D_4\simeq\F_3^4$ with the $3$-torsion subgroup of $A_T$, and $q_T(u/3+T)\equiv\frac{\sum_j u_j^2}{3}\pmod{2\Z}$.

The norm-$4$ vectors of $D_4$ are precisely $\pm2e_i$ and $(\pm1,\pm1,\pm1,\pm1)$. They represent all classes of quadratic value $\frac43$ in the $3$-torsion subgroup of $A_T$, which have one or four nonzero coordinates modulo $3$. These vectors form one $O(D_4)$-orbit: signed permutations are transitive within each type, and the reflection $u\mapsto u-\frac{\sum_j u_j}{2}(1,1,1,1)$ sends $2e_1$ to $(1,-1,-1,-1)$.

The following lemma gives the upper bound needed for $n=4$.

\begin{lemma}\label{lem:d4-projection}
Let $T\subset\Leech$ be the fixed lattice in row $114$ of \cite[Table~1]{HoehnMason}, so that $T\simeq D_4(3)$, and let $S=T^\perp_\Leech$. If $d\in A_S$ has order $3$ and $q_S(d)=\frac23$, then $\mu_S(d)\leq\frac{8}{3}$.
\end{lemma}

\begin{proof}
Let $t_1,\ldots,t_4$ be the fixed basis for row $114$ in entry \path{lattices[4,16][2]} of the supplement \path{lattices.txt} to \cite{HoehnMason}. The published ambient coordinates use scale $1/\sqrt8$. The basis $(\beta_1,\beta_2,\beta_3,\beta_4)=(-t_2,-t_1,t_4,t_1+t_3)$ is a simple-root basis of $D_4(3)$, with $\beta_2$ the central root. The coordinates needed below are
\[
 \begin{array}{c|rrrr}
 &\beta_1&\beta_2&\beta_3&\beta_4\\\hline
 \text{coordinate }1\text{ of }\sqrt8\,\beta_i&2&0&0&-1\\
 \text{coordinate }11\text{ of }\sqrt8\,\beta_i&-2&0&0&-1 .
 \end{array}
\]
Take the standard minimal Leech vector $\lambda=\frac{4E_1-4E_{11}}{\sqrt8}$, where $E_1,\ldots,E_{24}$ are the ambient orthonormal coordinate vectors. Then $\lambda^2=4$, and the table gives $((\lambda,\beta_i))_i=(2,0,0,0)$. The Gram matrix of the root basis gives $t=\pr_T(\lambda)=\frac{2\beta_1+2\beta_2+\beta_3+\beta_4}{3}$ of norm $\frac43$, so $\pr_S(\lambda)=\lambda-t$ has norm $\frac83$.
Since $3t\in T$ and $t\notin T$, its class has order $3$. Under unimodular gluing, $t+T\in A_T$ corresponds to $(\lambda-t)+S\in A_S$, and their quadratic values are opposite modulo $2\Z$.

The order-$3$ classes of value $\frac43$ in $A_T$ form one $O(T)$-orbit, as noted above. The normalizer index $1$ for row $114$ implies that these isometries extend to $\Leech$ preserving $T$ and $S$. Unimodular gluing therefore gives the bound for every order-$3$ class of value $\frac23$ in $A_S$.
\end{proof}

Lemma~\ref{lem:d4-projection} and Proposition~\ref{prop:depth} exclude $n=4$. For $n=7$, we use the rank-$4$ Mukai complement.

\begin{proposition}\label{prop:ternary}
Let $G$ be a finite group acting by symplectic automorphisms on a manifold $X$ of $K3^{[n]}$ type, and let $S$ be its associated Leech coinvariant lattice. If $n=4$ or $7$, then $\alpha(S)>0$.
\end{proposition}

\begin{proof}
Suppose $\alpha(S)=0$ and put $T=\Leech^G$. Let $v,R,a,c$ be as in Lemma~\ref{lem:gluing-class}. Proposition~\ref{prop:dimension} gives $a=\cont(T)=3$, so Lemma~\ref{lem:ternary-reduction} identifies $T$ with the row-$114$ embedding of $D_4(3)$.

If $n=4$, then Lemma~\ref{lem:gluing-class} gives a class $c$ of order $3$ with $q_S(c)=\frac23$. Proposition~\ref{prop:depth} and Lemma~\ref{lem:d4-projection} give $\frac{14}{3}\leq\mu_S(c)\leq\frac83$, a contradiction.

For $n=7$, one has $\rk S=20$, so $R$ is positive definite of rank $4$. Lemma~\ref{lem:gluing-class} and unimodular gluing give $\cont(R)=3$ and
\[
 |A_R|=|A_T|=|A_{D_4(3)}|=3^4|A_{D_4}|=4\cdot3^4.
\]
Since $R$ is even, the norms of its vectors are divisible by $6$; hence $R(\frac13)$ is even integral of determinant $4$.

For any positive-definite lattice $L$ of rank $4$, its minimum nonzero norm satisfies $\min L\leq\sqrt2\,(\det L)^{1/4}$ by the four-dimensional Hermite bound \cite[Chapter~1, Tables~1.1--1.2 and Eq.~(47)]{ConwaySloane}. Equality holds precisely for lattices similar to $D_4$ \cite[Chapter~1, \S~1.5, p.~13]{ConwaySloane}. Applying this bound gives $\min R(\frac13)\leq\sqrt2\,4^{1/4}=2$. Since $R(\frac13)$ is even, its minimum is at least $2$, so equality holds. As $\det D_4=4$ as well, the scale is fixed and $R(\frac13)\simeq D_4$. Therefore $R\simeq D_4(3)$.

Since $v^2=12$, the vector $v$ has norm $4$ in the unscaled $D_4$. A straightforward calculation of norm-$4$ vectors in $D_4$ shows that $v$ has divisibility exactly $2$ in $D_4$ (see the argument before Lemma~\ref{lem:d4-projection}). Hence $a=\divi_R(v)=6$, contradicting $a=3$.
\end{proof}

\subsection{The case \texorpdfstring{$n=5$}{n=5}}\label{sec:parity}

\begin{proposition}\label{prop:isotropic-two}
Let $G\subset O(\Leech)$ be a finite subgroup and let $S=S_G(\Leech)$. If $d\in A_S$ has order $2$ and $q_S(d)=0$, then $\mu_S(d)\in\{2,4\}$.
\end{proposition}

\begin{proof}
Choose a representative $x\in S^\vee$ of $d$, so $2x\in S$ and $x^2\in2\Z$. Unimodular gluing makes the action on $A_S$ trivial. Since $\Leech\cap S_\R=S$, the class $2x+2\Leech$ is nonzero and $G$-invariant. Its quadratic value is $2x^2\equiv0\pmod2$, so Lemma~\ref{lem:modtwo} gives a norm-$4$ pair or a norm-$8$ frame.

For a norm-$4$ pair $\{\pm w\}$, if some $g\in G$ sends $w$ to $-w$, averaging gives $w\in S$. Then $w/2-x\in\Leech\cap S_\R=S$, contradicting $q_S(d)=0$ since $(w/2)^2=1$. Otherwise $w$ is $G$-fixed, so $x\perp w$ and $x-w/2\in\Leech$ has the odd norm $x^2+1$, again impossible. 

Thus we have a frame $\{\pm w_i\}_{i=1}^{24}$. If $gw_i=-w_i$ for some $g,i$, the same averaging argument gives $w_i/2\in x+S$, of norm $2$. Otherwise each line stabilizer fixes its chosen vector, so we may choose signs along each orbit such that $G$ permutes the $w_i$.

Put $r_i=w_i/2$, so $(r_i,r_j)=2\delta_{ij}$ and $r_i-x\in\Leech$. Evenness of $(r_i-x)^2$ gives $(x,r_i)\in\Z$, and $(s,r_i)\in\Z$ for $s\in S$. Thus every $y\in x+S$ has the form $y=\frac12\sum_i m_i(y)r_i$, where $m_i(y)=(y,r_i)\in\Z$. We record their coordinate parities in the nonempty affine subspace
\[
 \mathcal A=\{(m_i(y)\bmod2)_{i=1}^{24}:y\in x+S\}\subset\F_2^{24}.
\]
We will find an element of $\mathcal A$ of weight at most $8$ and realize it by a representative of $d$ of norm at most $4$.

For each coordinate orbit $O$, the sum $\sum_{i\in O}r_i$ is $G$-fixed, so $\sum_{i\in O}m_i(y)=0$. In particular, every $b\in\mathcal A$ has an even number of nonzero coordinates in each orbit. Moreover, if $i,j\in O$, then $r_i-r_j\in S$: it belongs to $\Leech$, and its $G$-average is zero because $r_i$ and $r_j$ lie in the same orbit.

For $b\in\mathcal A$, choose a lift $y\in x+S$. In each orbit assign $+1$ to half the nonzero coordinates of $b$ and $-1$ to the other half, and assign $0$ elsewhere; write these integers as $\varepsilon_i$. Then
\[
 y_b=\frac12\sum_i\varepsilon_i r_i\in x+S.
\]
Indeed, $y-y_b=\sum_i\frac{m_i(y)-\varepsilon_i}{2}r_i$ has integer coefficients summing to zero on each orbit, hence is an integer combination of the differences $r_i-r_j\in S$. Thus $y_b$ still represents $d$, and its norm is half the number of nonzero coordinates of $b$. Consequently $0\notin\mathcal A$, and every weight in $\mathcal A$ is divisible by $4$, since all norms in $x+S$ are even.

Suppose $\sum_{i\in J}b_i=1$ for all $b\in\mathcal A$, for some $J\subset\{1,\ldots,24\}$ with $1\leq|J|\leq2$. Put $r_J=\sum_{i\in J}r_i$. Then $(x,r_J)$ is odd and $(s,r_J)$ is even for all $s\in S$, so $z=\frac12\pr_S(r_J)\in S^\vee$. Triviality of the action on $A_S$ gives $gz-z\in S$. Since $gr_J$ and $r_J$ have the same projection onto the fixed subspace, we have
\[
 gz-z=\frac12(gr_J-r_J),\qquad
 (gz-z)^2\leq|J|\leq2.
\]
The bound follows because cancelling common terms in $gr_J-r_J$ leaves at most $2|J|$ orthogonal terms of norm $2$. Since every nonzero vector of $S$ has norm at least $4$, we have $gz=z$ for every $g$. As $S_\R^G=0$, we get $z=0$, contradicting the oddness of $(x,r_J)=2(x,z)$.

Therefore $\mathcal A_0=\{b\in\mathcal A:b_1=0\}$ is nonempty. If a coordinate $b_j$ were constantly $1$ on $\mathcal A_0$, affine linearity would give either $b_j=1$ or $b_1+b_j=1$ throughout $\mathcal A$, both excluded above. Each coordinate on $\mathcal A_0$ is thus either constantly $0$ or takes $0$ and $1$ equally often. The first coordinate is constantly $0$, so the average weight is at most $23/2<12$.

Some $b\in\mathcal A_0$ therefore has weight $4$ or $8$, giving $y_b\in x+S$ of norm $2$ or $4$. Since $d\ne0$ and $q_S(d)=0$, this proves $\mu_S(d)\in\{2,4\}$.
\end{proof}

Combining Propositions~\ref{prop:depth} and~\ref{prop:isotropic-two} excludes $n=5$:

\begin{proposition}\label{prop:five}
Let $G$ be a finite group acting by symplectic automorphisms on a manifold $X$ of $K3^{[5]}$ type, and let $S$ be its associated Leech coinvariant lattice. Then $\alpha(S)>0$.
\end{proposition}

\begin{proof}
Suppose $\alpha(S)=0$. Let $v,R,a,c$ be as in Lemma~\ref{lem:gluing-class}. Then $a>1$ and $a\mid8$, so $a$ is even. The class $d=\frac{a}{2}c$ has order two and $q_S(d)=0$ by Lemma~\ref{lem:gluing-class}\eqref{eq:glueclass}. Proposition~\ref{prop:depth}, applied with $k=a/2$, and Proposition~\ref{prop:isotropic-two} give $6\leq\mu_S(d)\leq4$, a contradiction.
\end{proof}

We now come to the proof of the main theorem.

\begin{proof}[Proof of Theorem~\ref{thm:main}]
As noted in \S\ref{sec:coinvariant-wall}, $\rk S\leq20$. Proposition~\ref{prop:geometry} shows that the embedding associated with $X$ satisfies the hypothesis of Proposition~\ref{prop:dimension}. If $\alpha(S)=0$, Proposition~\ref{prop:dimension} leaves $n\in\{2,3,4,5,7\}$. These cases are excluded by Propositions~\ref{prop:low}, \ref{prop:ternary}, and \ref{prop:five}.
\end{proof}

Combining Theorem~\ref{thm:main} and Lemma~\ref{lem:leech-embedding} with the realization results proved independently by Huybrechts \cite[Theorem~0.2]{HuybrechtsConway} and Mongardi \cite[Theorem~29 and Lemma~30]{MongardiClassification} gives the following existence criterion:

\begin{corollary}\label{cor:criterion}
Let $G$ be a finite group acting faithfully by isometries on an even lattice $S$. There exist an integer $n\geq2$ and a manifold $X$ of $K3^{[n]}$ type with a faithful symplectic $G$-action such that $S\simeq S_G(X)(-1)$ as $G$-lattices if and only if $(S,G)$ is a Leech pair as defined in \S\ref{sec:leech}, $\rk S\leq20$, and $\rk S+\ell(A_S)<24$.
\end{corollary}

\Addresses
\end{document}